\documentclass{amsart}

\usepackage{amsmath}

\usepackage[pagebackref=false,colorlinks=true, pdfstartview=FitV, linkcolor=ComplementFG,citecolor=ForestGreen, urlcolor=blue]{hyperref}

\usepackage{mathtools}

\usepackage{xcolor}
\colorlet{Violet}{red!45!blue}
\usepackage{soul}

\definecolor{ForestGreen}{RGB}{34,139,34}
\definecolor{ComplementFG}{RGB}{34,86,139}

\usepackage{amsmath,amsthm,amssymb,amsfonts}
\usepackage{verbatim}
\usepackage{marginnote}

\newtheorem{theorem}{Theorem}[section]

\newtheorem{proposition}[theorem]{Proposition}
\newtheorem{lemma}[theorem]{Lemma}

\newtheorem*{theorem*}{Theorem}
\theoremstyle{definition}
\newtheorem{remark}[theorem]{Remark}
\newtheorem{definition}[theorem]{Definition}

\newcommand{\frp}{\operatorname{frac}}

\numberwithin{equation}{section}

\title{The Three Gap Theorem and Periodic Functions}
\author{A. Suki Dasher}
\address{University of Minnesota, Minneapolis, MN, USA}
\email{dasher@umn.edu}

\author{A. Hermida}
\address{Smith College, Northampton, MA, USA}
\email{ahermida@uw.edu}

\author{Tian An Wong}
\address{University of Michigan, Dearborn, MI, USA}
\email{tiananw@umich.edu}

\date{\today}

\begin{document}

\begin{abstract}
The Three Gap Theorem, also known as the Steinhaus Conjecture, is a classical result on the combinatorics of the fractional part function, and has since been generalized in many ways.
In this paper, we pose a new problem related to these results:
for which other periodic functions does an analogue of the Three Gap Theorem hold?
We prove analogous results for certain classes of piecewise-linear periodic functions and demonstrate the existence of functions for which no bound exists on the number of gap lengths.
\end{abstract}

\maketitle

\section{Introduction}
The Three Gap Theorem, originally conjectured by H. Steinhaus, states that, given an angle $\theta$ and a positive integer $N$, the points corresponding to the angles $\theta, 2\theta, \ldots, N\theta$ on a directed circle partition its circumference into arcs of at most three distinct lengths.
In particular, for any real $\alpha$ and positive integer $N$, the elements of the set
\begin{equation}
\label{nalpha}
\{d\alpha \bmod 1 : 1 \leq d \leq N\}
\end{equation}
partition the circle $\mathbb{R}/\mathbb{Z}$ into subintervals of at most three different lengths, where the length of a subinterval $[a,b]$ is given by $\min \{|b-a+n| : n \in \mathbb Z\}$.
This result is in contrast to the fact that for irrational $\alpha$, the points $$\alpha, 2\alpha, \ldots, N\alpha$$ taken modulo 1 become equidistributed in the circle as $N$ tends to infinity, a well-known result due to Weyl.
Many proofs of the Three Gap Theorem have been given \cite{CG, halton, liang, slater, sos, swier, ravenstein} and the theorem has been generalized in multiple forms, including to Riemannian manifolds \cite{BS} and ergodic theory \cite{chevallier, gen0}.
Recent works continue to generalize and produce new proofs of the theorem \cite{BGS,BK,DasHaynes,manish}.
Interpreting the elements in \eqref{nalpha} as values of the fractional part function
\begin{equation}
\label{frac}
\frp(x) = x - \lfloor x \rfloor
\end{equation}
from $\mathbb{R}$ to the circle $S^1 \simeq \mathbb R/\mathbb Z$, we can place the Steinhaus problem in the context of a more general question:
given a periodic function $f  : \mathbb{R} \to \mathbb{R}$, does there exist a bound on the number of distinct distances between adjacent points of the set
$$
\{f(d\alpha) : 1 \leq d \leq N\}
$$
which is independent of $\alpha$ and $N$?

Let $G_{f,\alpha,N}$ be the set of distinct gap lengths determined by the elements of the above set (see Definition \ref{gaps} for a precise description) and let $|G_{f,\alpha,N}|$ be its cardinality.
In this paper, we explore the existence of a bound on $|G_{f,\alpha,N}|$ for various classes of periodic functions $f$ and show that $|G_{f,\alpha,N}|$ is indeed bounded in certain cases where $f$ is piecewise-linear.

In our study of periodic functions $f$, we fix conventions by assuming $P > 0$ is the fundamental period of $f$ and restrict our discussion to the values of $f$ on the interval $[0,P)$.
To avoid pathological cases, we consider only piecewise-linear periodic functions which satisfy the following.
We may partition $[0,P)$ into finitely many disjoint intervals $I_1, I_2,\dots, I_\ell$ of the form $I_i = [a_i, b_{i})$ where $a_i < b_i$ such that $f$ is linear over each
interval $I_i$ and 
$$
[0,P) =\bigsqcup_{i=1}^{\ell} I_i.
$$
We assume the ordering on the subintervals is increasing in the sense that given $I_{i_1}$ and $I_{i_2}$ with $1\le i_1 < i_2 \le \ell$, the right endpoint of $I_{i_1}$ is less than or equal to the left endpoint of $I_{i_2}$.
We also assume the partition  is maximal in the sense that for any $I_i$ with $1\le i < \ell$, the function $f$ is linear on $I_i$ and $I_{i+1}$, but non-linear on
$I_i \cup I_{i+1}$.
If $f$ is a piecewise-linear function, each $I_i$ corresponds to a linear equation $f_i$ in the definition of $f$ over its fundamental domain $[0,P)$, so we call $f_i$ a {\em piece} of $f$.
The number of pieces of $f$ is $\ell$, which is the minimal number of linear equations necessary to define $f$ over $[0,P)$.
We say that $f$ is \textit{injective on its fundamental domain} if it is injective on $[0, P)$.
Our first main result is the following theorem.

\begin{theorem}
	\label{general}
	Let $f$ be a piecewise-linear periodic function which is injective on its fundamental domain and has $\ell$ pieces whose slopes have $\mu$ distinct magnitudes.
	Then for any $\alpha \in \mathbb{R}$ and $N \in \mathbb{N}$, 
	$$
	|G_{f,\alpha,N}| \leq 3\mu + \ell.
	$$ 
\end{theorem}

\noindent
For piecewise-linear periodic functions which are not injective on their fundamental domains, under some hypotheses $|G_{f, \alpha, N}|$ is bounded (for examples, see Section \ref{PLNPI}).
However, one cannot in general find a bound on $|G_{f, \alpha, N}|$ that is dependent upon the number of pieces of $f$ and their corresponding slopes as in the above theorem.

\begin{theorem}
\label{main}
For any $n \in \mathbb{N}$, there exists a piecewise-linear periodic function $f$ which has two pieces and such that for some $\alpha \in \mathbb{R}$ and some $N \in \mathbb{N}$,
$$
|G_{f, \alpha, N}| > n.
$$ 

\end{theorem}

\noindent We prove an analogous result for certain non-linear functions using a similar method. 

\begin{theorem}
\label{sin}
Let $f$ be a periodic function which is twice continuously differentiable and such that $f''(0) \neq 0$.
Then for any $n \in \mathbb{N}$, there exists $\alpha \in \mathbb{R}$ such that
$$
|G_{f,\alpha,n+1}| \geq n.
$$
\end{theorem}

The above three theorems lead to the question of characterizing the set of periodic functions $f$ such that $G_{f,\alpha,N}$ is finite for any $\alpha$ and $N$, which we leave as a possible direction for future study.
More generally, we observe that the perspective taken in this paper can also be applied to other generalizations of the Three Gap Theorem, as a question about the dynamics of self-maps of a given space, and the nature of such finiteness results.

We conclude this introduction with a brief summary of the paper.
In Section \ref{affine1} we introduce the basic definitions and notations that we require, and also extend the Three Gap Theorem to affine transformations of the fractional part function $\frp(x)$ given in (\ref{frac}).
Then in Sections \ref{PLPI} and \ref{PLNPI} we study the gap length sets of piecewise-linear periodic functions that are injective and not injective on their fundamental domains, respectively, and prove Theorems \ref{general} and \ref{main}.
Finally, in Section \ref{NL} we show that gap length sets of certain non-linear periodic functions can be arbitrarily large, proving Theorem \ref{sin}.

\section{Affine Transformations}
\label{affine1}

To discuss the number of distinct gap lengths of a function, we first introduce the notion of a gap length set associated to a periodic function.

\begin{definition}
\label{gaps}
Let $f$ be a bounded periodic function.
For any $\alpha \in \mathbb{R}$ and $N \in \mathbb{N}$, let $s_1, s_2, \dots, s_n$ denote the distinct values of 
$$
f(d\alpha),\quad  1 \leq d \leq N,
$$
indexed such that $s_1 < s_2 < \cdots < s_n$.
Then the \textit{gap length set} of $f$ associated with $\alpha$ and $N$ is the set 
$$
G_{f,\alpha,N} = \{s_1 - \inf{f} + \sup{f} - s_n \} \cup \{ s_{j+1}-s_j : 1 \leq j < n \},
$$
the elements of which are called \textit{gap lengths}.
Additionally, let $\beta \in \mathbb{R}$ and let $t_1, t_2, \dots, t_m$ denote the distinct values of 
$$
f(d\alpha + \beta), \quad 1 \leq d \leq N,
$$
indexed such that $t_1 < t_2 < \cdots < t_m$.
Then we define 
$$
G_{f,\alpha, \beta, N}= \{t_1 - \inf{f} + \sup{f} - t_m \} \cup \{ t_{j+1}-t_j : 1 \leq j < m \}.
$$
\end{definition}
\begin{remark}
It follows immediately from the definition that $|G_{f,\alpha,N}| = |G_{f,\alpha,0,N}|$.
Furthermore, for any $f,\alpha,\beta$, and $N$ we have $|G_{f,\alpha,\beta,N}| = |G_{c_1f+c_2,\alpha,\beta,N}|$ for constants $c_1 \neq 0$ and $c_2$. 
\end{remark}

To make precise the notion of adjacent points in the image of $f$, we introduce the following definition.

\begin{definition}
	Let $f$ be a function and let $x_1, x_2, \ldots, x_n$ be a sequence of points in its domain.
	If $f(x_i) \neq f(x_j)$ for $1 \leq i,j \leq n$, then $f(x_i)$ and $f(x_j)$ are called \textit{nearest neighbors} if the set
	$$
	\Big\{f(x_k) : \min\{f(x_i),f(x_j)\} \leq f(x_k) \leq \max\{f(x_i),f(x_j)\}, \,\ 1 \leq k \leq n\Big\}
	$$
	is equal to either $\{f(x_i),f(x_j)\}$ or $\{f(x_1), f(x_2), \ldots, f(x_n)\}$.
	In the second case, we call the pair $f(x_i)$ and $f(x_j)$ \textit{extremal nearest neighbors}.
\end{definition}

A natural generalization of the Three Gap Theorem is to affine transformations of the fractional part function $\frp(x)$.
Without loss of generality, one may assume that a function has period 1 by replacing a function $f(x)$ with period $P$ by $f(\frac{x}{P})$.
We may therefore restrict our attention to functions with period 1 as in the following proposition, which is a direct consequence of the Three Gap Theorem. 

\begin{proposition}
\label{affine}
Let $f$ be a function with period $1$ given by $f(x) = mx + c$ on its fundamental domain, where $m, c \in \mathbb{R}$.
Then for any $\alpha, \beta\in\mathbb{R}$ and $N \in \mathbb{N}$,
$$
|G_{f,\alpha, \beta, N}| \leq 3.
$$
\end{proposition}

\begin{proof}
If $m=0$ the theorem is trivially true, and we may further assume $m > 0$ as the proof is similar for $m < 0$.
Let $s_1, s_2, \ldots, s_n$ denote the distinct values of 
$$
\frp(d\alpha + \beta),\quad 1 \leq d \leq N,
$$
indexed such that $s_1 < s_2 < \dots < s_n$.
By the Three Gap Theorem, there exist $v_1, v_2, v_3 \in \mathbb{R}$ (not necessarily distinct) such that for each $1\leq j<n$, we have $s_{j+1} - s_j = v_i$ for some $1\le i\le 3$, and $s_1 - \inf{\frp} + \sup{\frp} - s_n = v_k$ for some  $1\le k\le 3$.
	
The values of $f(d\alpha + \beta)$ for all $1 \leq d \leq N$ are of the form $ms_j + c$ for $1 \leq j \leq n$.
Since $m > 0$, the values $ms_j + c$  are distinct and may be relabeled as $t_j = ms_j + c$.
To determine the cardinality of $G_{f,\alpha,\beta,N}$, we compute $t_{j+1} - t_j$ for $1 \leq j < n$ and $t_1 - \inf{f} + \sup{f} - t_n$. Since $t_{j}$ and $t_{j+1}$ are nearest neighbors, $s_{j}$ and $s_{j+1}$ must be nearest neighbors.
Therefore, 
	\begin{align*}
	t_{j+1} - t_j &= ms_{j + 1} + c - ms_{j} - c 
	= m(s_{j + 1} - s_{j}) 
	=mv_i 
	\end{align*}
	for some $1 \leq i \leq 3$.
	Furthermore, $\sup{f} = m \sup{\frp} + c$ and $\inf{f} = m \inf{\frp} + c$, so
	\begin{align*}
	t_1 - \inf{f} + \sup{f} - t_n 
	=m(s_1 - \inf{\frp}  + \sup{\frp} - s_n)
	=mv_k 
	\end{align*}
	for some $1 \leq k \leq 3$.
	This demonstrates the elements of $G_{f,\alpha,\beta,N}$ are of the form $mv_i$ for $1 \leq i \leq 3$, from which it follows that $|G_{f,\alpha,\beta,N}| \leq 3$. 
\end{proof}

\begin{remark}
Intuitively, the finiteness of the gap length set associated to a function should be independent of affine transformations, but this does not follow from a trivial change of variables.
Indeed, if $f$ is a periodic function and $r$ is a real number, it is not necessarily true that $G_{f(x), \alpha, N}$ and $G_{f(x-r), \alpha, N}$ are equal, and in fact the gap length sets may be disjoint.
For example, if $\alpha = \frac{1}{4}$ and one evaluates $\cos(x)$ at $\alpha, 2\alpha, 3\alpha$, then the resulting gap lengths are approximately $0.0913$, $0.1459$, and $1.7628$.
If the shifted function $\cos(x - \frac{\pi}{2})$ is evaluated at $\alpha, 2\alpha, 3\alpha$ as well, the gap lengths are approximately $0.2022$, $0.2320$, and $1.5658$.
This example illustrates that the hypothesis $f''(0) \neq 0$ in Theorem \ref{sin} cannot be dropped without loss of generality by assuming the function $f$ may be horizontally translated.
\end{remark}

\section{Injective Piecewise-Linear Functions}
\label{PLPI}

We now turn our attention to piecewise-linear periodic functions which are injective on their fundamental domains.

\begin{definition}
Given $\alpha \in \mathbb{R}$, $N \in \mathbb{N}$, and a piecewise-linear periodic function $f$, let $s$ and $s'$ be elements of $\{f(d\alpha) : 1 \leq d \leq N \}$ which are nearest neighbors such that $s < s'$.
If $s=f_i(d\alpha)$ and $s'=f_i(d'\alpha)$ for some piece $f_i$ of $f$ and $1 \leq d, d' \leq N$ such that $\frp(d\alpha)$ and $\frp(d'\alpha)$ are elements of $I_i$, then we call $s' - s$
an \textit{interior gap length} of the piece $f_i$, or more generally, of $f$.
When this is not the case, we call $s' - s$ a \textit{non-interior gap length}.
If $s$ and $s'$ are extremal nearest neighbors, then we call
$$
s - \inf{f} + \sup{f} - s'
$$
the \textit{extremal gap length} of $f$.
\end{definition}
\noindent
By definition, there is a unique extremal gap length of $f$ once $\alpha$ and $N$ are fixed.
However, it is not necessarily distinct from all other gap lengths of $f$.
Note also that a non-interior gap length is not necessarily the extremal gap length.

\begin{lemma}
\label{Interior}
Let $f$ be a piecewise-linear periodic function with period 1 which is injective on its fundamental domain.
Suppose also that $f$ has distinct pieces $f_1$, $f_2$, \ldots, $f_k$ with slopes $|m_1| = |m_2| = \cdots = |m_k| = m$ defined on subintervals $I_1, I_2, \ldots, I_k$, respectively, as in the discussion at the beginning of the section.
Then for all $\alpha \in \mathbb{R}$ and $N \in \mathbb{N}$, the union of the sets of the interior gap lengths of $f_1$, $f_2$, \ldots, $f_k$ has at most three elements.
In particular, each piece of $f$ has at most three interior gap lengths.
\end{lemma}

\begin{proof}
Suppose that a piece of $f$ is given by  $f_j(x) = m_jx+ c_j$.
Let $d$ and $d'$ be distinct integers with $1\le d,d' \le N$ such that $\frp(d\alpha), \frp(d'\alpha) \in I_j$, and $f(d\alpha) = f_j(d\alpha)$ and $f(d'\alpha) = f_j(d'\alpha)$ are nearest neighbors.
(Note that if no such $d$ and $d'$ exist, then the lemma is vacuously true.)
Then by the same reasoning as in the proof of Proposition \ref{affine}, $\frp(d\alpha)$ and $\frp(d'\alpha)$ are nearest neighbors.
By the Three Gap Theorem there exist $v_1, v_2, v_3 \in \mathbb{R}$ (not necessarily distinct) which are independent of $d$, $d'$ such that 
$$|\frp(d'\alpha) - \frp(d\alpha)| = v_i$$
for some $1\le i \le 3$.
Thus, we may explicitly compute the interior gap length between $f(d\alpha)$ and $f(d'\alpha)$ as  
	\begin{align*}
	|f_j(d'\alpha) - f_j(d\alpha)| &= |m_j \frp(d'\alpha) + c_j - m_j \frp(d\alpha) - c_j|\\
	&=|m_j(\frp(d'\alpha) - \frp(d\alpha))| \\
	&=|m_j||\frp(d'\alpha) - \frp(d\alpha)| \\
	&=|m_j| v_i \\
	&= m v_i.
	\end{align*}
This gap length is independent of the choice of $j$, which concludes the proof.
\end{proof}

We may now prove the general statement for piecewise-linear periodic functions that are injective on their fundamental domains.

\begin{proof}[Proof of Theorem \ref{general}]
	The maximal number of gap lengths of $f$ is achieved when the maximal number of both interior gap lengths and non-interior gap lengths is achieved, so we may simply count the possibilities to determine an upper bound for the cardinality of the gap length set.
	By Lemma \ref{Interior}, the interior gap lengths of pieces whose slopes have a common magnitude are all members of the same set which has at most three elements.
	Therefore, the maximal number of distinct interior gap lengths from all pieces of $f$ is $3\mu$.
	There are at most $\ell$ gap lengths between nearest neighbors of the form $f(d\alpha) = f_i(d\alpha)$ and $f(d'\alpha) = f_j(d'\alpha)$ for some $1 \leq d, d' \leq N$ such that $\frp(d\alpha) \in I_i$ and $\frp(d'\alpha) \in I_j$ for $i \neq j$.
	Adding the number of possibilities yields an upper bound of $3\mu + \ell$.
\end{proof}

	\label{twopiece}
	The above bound is sharp, but in special cases may be tightened.
	For example, if $f$ satisfies the hypotheses of Theorem \ref{general}, is monotonic on its fundamental domain, and the slopes of its pieces $f_1$ and $f_\ell$ are equal, then $|G_{f, \alpha, N}| \leq 3\mu + \ell - 1$.
	To see this, observe that the maximal number of gap lengths of $f$ is achieved when the maximal number of both interior gap lengths and non-interior gap lengths occur.
	We shall demonstrate that in this case, the extremal gap length is equal to an interior gap length, thus tightening the bound by 1.
	
	Let the first and last pieces of $f$ be given by $f_1(x) = mx + c_1$ and $f_\ell(x) = mx + c_\ell$, and let $\alpha$ and $N$ be numbers such that the maximal number of distinct gap lengths of $f$ is realized.
	Let $g$ be the periodic function with the same fundamental domain as $f$, defined by a single piece $g(x) = mx + c_1$.
	Observe that the interior gap lengths of $f_1$, $f_\ell$, and $g$ are equal.
	If $f_1$ has three distinct interior gap lengths, then so must $g$, in which case the extremal gap length of $g$ is equal to one of its interior gap lengths by Proposition \ref{affine}.
	Thus, it suffices to show that the extremal gap lengths of $f$ and $g$ are equal, which follows from the fact that vertical translation preserves distances.
	More explicitly, let $s_1 < \cdots < s_n$ and $t_1 < \cdots <t_n$ denote the distinct values of $f(k\alpha)$ and $g(k\alpha)$, respectively, for $1 \leq k \leq N$.
	Observe that if $f$ is increasing, then $s_1 = t_1$, $\inf f = \inf g$, $s_n = t_n - c_1 + c_\ell$, and $\sup f = \sup g - c_1 + c_\ell$.
	Similarly, if $f$ is decreasing, then $s_n = t_n$, $\sup f = \sup g$, $s_1 = t_1 + c_1 - c_\ell$, and $\inf f = \inf g + c_1 - c_\ell$.
	In either case, $\sup f - s_n = \sup g - t_n$ and $s_1 - \inf f = t_1 - \inf g$, so $\sup f - s_n + s_1 - \inf f= \sup g - t_n + t_1 - \inf g$, and thus the extremal gap lengths of $f$ and $g$ are equal.
	
	The simplest illustration of this tightened bound is the Three Gap Theorem itself.
	Note that a tighter bound does not necessarily hold for functions which are not monotonic on their fundamental domains.
	Consider, for example, the periodic function $f$ whose definition on its fundamental domain $[0,1)$ is
	$$
	f(x) = 
	\begin{cases}
	x + 1 & 0 \leq x < \frac{3}{4} \\
	x - \frac{1}{2} & \frac{3}{4} \leq x < 1
	\end{cases}
	$$
	and let $\alpha = \frac{\pi}{16}$ and $N = 7$.
	Then one may compute $|G_{f, \alpha, N}| = 5$.
	
\section{Arbitrary Piecewise-Linear Functions}
\label{PLNPI}

In this section we turn our attention to piecewise-linear functions which are not injective on their fundamental domains.
A simple example is the triangle wave, given by the function $f(x) = [x]$, which returns the distance from $x$ to the nearest integer.
The analogue of the Three Gap Theorem for the triangle wave, due to H. Don \cite[Theorem 2]{don} states that $2 \le |G_{f,\alpha,N}| \le 4$.
For functions with pieces whose slopes do not have the same magnitude, one may consider the gap lengths formed over intervals on which a function is not injective by recalling the following result due to P. Alessandri and V. Berth\'e \cite[Theorem 18]{berthe}.

\begin{lemma}
	\label{AB}
	Let $\alpha \in (0,1)$ be irrational, let $\beta \in \mathbb{R}$ be non-zero, and let $N \in \mathbb{Z}$ be non-zero.
	Then the elements of $$\{\frp(d\alpha)\}\cup \{\frp(d\alpha+\beta)\}$$ for $d = 0, \ldots, N$ partition the circle into a finite number of intervals whose lengths take at most five distinct values.	
\end{lemma}

\noindent The result immediately extends to all real $\alpha$, which we use in the proof of the following proposition.

\begin{proposition}
	Let $f$ be a periodic  function defined on its fundamental domain by 
	$$f(x) =  \begin{cases} 
	x & 0\leq x < \kappa \\
	x - \beta & \kappa \leq x < 1 
	\end{cases} $$
	where $0<\beta \leq \kappa < 1$. Then for any $\alpha \in \mathbb{R}$ and $N \in \mathbb{N}$, we have $|G_{f,\alpha,N}| \leq 10.$
\end{proposition}

\begin{proof}
Partition $[0,1)$ into $I_1 = [0,\kappa)$ and $I_2 = [\kappa,1)$, which correspond to the pieces $f_1(x) = x$ and $f_2(x) = x - \beta$, respectively. Then the image of $f$ may be expressed as the disjoint union
$$
U_1\sqcup U_2\sqcup U_{12} \coloneqq \big(f(I_1) \setminus f(I_2) \big) \sqcup \big( f(I_2) \setminus f(I_1) \big) \sqcup \big( f(I_1) \cap f(I_2) \big).
$$
The maximal number of gap lengths of $f$ is achieved when in each of the intervals in the above partition of $f([0,1))$, the maximal number of gap lengths occur.
In this case, there are at most two distinct gap lengths between nearest neighbors $s$ and $s'$ where $s \in U_{12}$ and $s' \in U_i$ for $i=1,2$.
Thus, the proof reduces to a matter of counting the number of possible gap lengths between nearest neighbors $s$ and $s'$ in the cases where (i) $s, s' \in U_{12}$ and (ii) $s, s' \in U_i$ or $s \in U_1$ and $s' \in U_2$.

In the case where $s, s' \in U_{12}$, then $s$ and $s'$ are of the form of the points in Theorem \ref{AB}, so the interval between them may have one of at most five lengths.
(Note that by assuming the maximal number of gap lengths have occurred, $s$ and $s'$ are necessarily non-extremal nearest neighbors.)
Similarly, the distance between nearest neighbors $s, s' \in U_i$ is an interior gap length of $f_i$, so it follows from Lemma \ref{Interior} that the interval between them may have one of at most three lengths.

It remains only to consider nearest neighbors $s$ and $s'$ where $s \in U_1$ and $s' \in U_2$.
If the maximal number of gap lengths occur over each interval in the given partition of $f([0,1))$, then if there exists a gap length between nearest neighbors $s \in U_1$ and $s' \in U_2$, it is necessarily the extremal gap length of $f$.
By a similar argument as given after the proof of Theorem \ref{general}, the extremal gap length of $f$ is a possible interior gap length.
Adding the possibilities yields a total of at most ten distinct gap lengths. 	
\end{proof}

Though the above results give examples of functions with gap length sets with bounded cardinalities, there does not exist an analogue of Theorem \ref{general} for arbitrary piecewise-linear periodic functions, even for those with only two pieces.
For any $n \in \mathbb{N}$, a piecewise-linear periodic function $f$ with two pieces may be constructed such that for judiciously chosen $\alpha \in \mathbb{R}$ and $N \in \mathbb{N}$, the cardinality of $G_{f, \alpha, N}$ exceeds $n$.
We illustrate this in the following proof. 

\begin{proof}[Proof of Theorem \ref{main}]
	Let $n \in \mathbb{N}$ and let $N$ be an even natural number greater than 4 such that $n < \frac{N}{2} - 1$. Let $0 <\varepsilon < \frac{2}{N-4}$ and let $f$ be the periodic function defined by
	$$
	f(x) =
	\begin{cases}
	x & 0 \leq x \leq \frac{1}{2} \\
	(1 + \varepsilon)x - \frac{1 + \varepsilon}{2} & \frac{1}{2} < x < 1
	\end{cases}
	$$
	on its fundamental domain.
	Then for $\alpha = \frac{1}{N}$, the first $\tfrac{N}{2}$ terms and the last term of the sequence $\alpha, 2\alpha, \ldots, N\alpha$ are evaluated by the piece $f_1(x) = x$ and the remaining terms are evaluated by
	$f_2(x) = (1 + \varepsilon)x - \frac{1 + \varepsilon}{2}
	$.
	For all integers $a$ satisfying 
	$2 \leq a \leq \tfrac{N}{2} - 1$,
	$$
	\frac{a - 1}{N} + \frac{(a - 1)\varepsilon}{N} < \frac{a}{N}. 
	$$
	This inequality shows that the sequence of images $f(\alpha), f(2\alpha), \ldots, f(N\alpha)$ is rearranged in increasing order as
	$$
	0, \,\frac{1}{N}, \, \frac{1}{N} + \frac{\varepsilon}{N}, \, \frac{2}{N}, \, \frac{2}{N} + \frac{2\varepsilon}{N}, \, \ldots,  \frac{\tfrac{N}{2}-1}{N}, \, \frac{\tfrac{N}{2} - 1}{N} + \frac{(\tfrac{N}{2} - 1)\varepsilon}{N}, \, \frac{\tfrac{N}{2}}{N}.
	$$
	A subset of the gap lengths may be calculated by taking the differences of consecutive terms in the sequence which are of the form $\frac{k}{N}, \, \frac{k}{N} + \frac{k\varepsilon}{N}$ for $1 \leq k \leq \frac{N}{2} - 1$.
	This yields $\frac{N}{2}-1$ distinct gap lengths:
	$$\frac{\varepsilon}{N}, \frac{2\varepsilon}{N}, \ldots, \frac{(\tfrac{N}{2} - 1)\varepsilon}{N}.$$ Since $\frac{N}{2} - 1 > n$, we conclude that $|G_{f,\alpha,N}| > n$.
\end{proof}

\section{Non-Linear Functions}
\label{NL}

Using a method similar to that in the proof of Theorem \ref{main}, one can show that a bound on the cardinality of the gap length set does not exist for certain periodic functions, for example, $\sin(x)$, even if such a function is restricted to a domain on which it is injective.
We prove this for a class of functions satisfying mild hypotheses. 

\begin{proof}[Proof of Theorem \ref{sin}]
	Let $f$ be a $C^2$ function with period $P$ such that $f''(0) \neq 0$, and let 
	\[
	\mathcal{I} = \inf \{P \geq x \geq 0 : f ''(x) = 0 \}.
	\]
	By continuity of $f''$, the preimage of $\{0\}$ under $f''$ is a closed set, so 
	\[
	(f'')^{-1}(\{0\}) \cap [0,P] = \{P \geq x \geq 0 : f ''(x) = 0 \}
	\]
	is compact and thus contains its infimum.
	It follows from $f''(0) \neq 0$ that $\mathcal{I} \neq 0$.
	Since $f''$ is continuous, $f''(x) \neq 0$ for $x \in [0,\mathcal{I}]$ implies $f''$ is either positive or negative on $[0,\mathcal{I}]$, so $f'$ is strictly monotone and thus injective when restricted to $[0,\mathcal{I}]$.
	Now let 
	\[
	\mathcal{I}' = \inf \{ \mathcal{I} \geq x \geq 0 : f'(x) = 0\}.
	\]
	By the argument above, $\mathcal{I}'$ is contained in the compact set $(f')^{-1}(\{0\}) \cap [0, \mathcal{I}]$.
	Since $f'$ is injective and continuous on $[0, \mathcal{I}]$, if $\mathcal{I}' = 0$ then $f'$ is either positive or negative on $[0, \mathcal{I}]$, and if $\mathcal{I}' \neq 0$, then $f'$ is either positive or negative on $[0, \mathcal{I}']$.
	This implies $f$ is strictly monotone and thus injective on $[0,\mathcal{I}]$ or $[0, \mathcal{I}']$, respectively.
	
	Given any $n \in \mathbb{N}$, let $\alpha = \frac{\mathcal{I}}{n + 1}$ if $\mathcal{I}' = 0$ or $\alpha = \frac{\mathcal{I}'}{n + 1}$ if $\mathcal{I}' \neq 0$.
	Since the restriction of $f$ to $[0, \mathcal{I}]$ or $[0, \mathcal{I}']$, respectively, is injective, the points 
	\[
	f(\alpha), f(2\alpha), \ldots, f((n+1)\alpha)
	\]
	are distinct and in either ascending or descending order, which implies 
	\[
	\{|f((k+1)\alpha) - f(k\alpha)| : 1 \leq k \leq n\} \subset G_{f, \alpha, n+1}.
	\]
	Now, the gap lengths in this subset are values of the function $|f((x+1)\alpha) - f(x\alpha)|$.
	The derivative 
	\[
	\alpha f'((x + 1)\alpha) - \alpha f'(x\alpha)
	\]
	 of $f((x+1)\alpha) - f(x\alpha)$ is non-zero for $x \in [0,n]$ since $f'$ is injective on its restricted domain, so by continuity, $\alpha f'((x + 1)\alpha) - \alpha f'(x\alpha)$ is either positive or negative on $[0,n]$.
	This implies $f((x+1)\alpha) - f(x\alpha)$ on $[0,n]$ is strictly monotone and either positive or negative since $f$ is strictly monotone on $[0,\mathcal{I}]$.
	Thus, $|f((x+1)\alpha) - f(x\alpha)|$ is injective on $[0,n]$.
	In particular, the values of 
	\[
	|f((k+1)\alpha) - f(k\alpha)|,\qquad 1 \leq k \leq n,
	\]
	are distinct and therefore $|G_{f, \alpha, n+1}| \geq n$.
\end{proof}

\subsection*{Acknowledgments} The authors acknowledge the support of the Center for Women in Mathematics, Smith College. The authors thank Sarah Brauner for feedback on the manuscript, Andrew O'Desky for crucial suggestions, and the referee at \textit{Integers} for comments improving the main results and exposition of the paper.
The third named author also thanks Manish Mishra for suggesting this problem. 
\bibliography{3GT}
\bibliographystyle{alpha}
\end{document}